\documentclass[1p]{elsarticle}
 
 \usepackage[latin1]{inputenc}
\journal{MATRIX Annals}

\usepackage{amssymb}
\usepackage{eucal}
\usepackage{graphicx}
\usepackage{tikz}
\usepackage{subcaption}
\usetikzlibrary{fadings}
\usetikzlibrary{patterns}
\usetikzlibrary{shadows.blur}
\usetikzlibrary{shapes}
\usepackage[normalem]{ulem}
\newcommand{\R}{{\Bbb R}}

\usepackage{color}

\usepackage{xcolor}
\definecolor{alizarin}{rgb}{0.82, 0.1, 0.26}
\definecolor{burgundy}{rgb}{0.5, 0.0, 0.13}
\definecolor{darkblue}{rgb}{0.0, 0.0, 0.55}
\definecolor{deepcarmine}{rgb}{0.66, 0.13, 0.24}
\definecolor{navyblue}{rgb}{0.0, 0.0, 0.5}

\usepackage{amsmath}

\newtheorem{thm}{Theorem}
\newtheorem{lem}[thm]{Lemma}
\newtheorem{cor}[thm]{Corollary}

\newtheorem{remark}[thm]{Remark}

\newproof{proof}{Proof}
\begin{document}
\begin{frontmatter}

\title{On the global dynamics of a forest model with monotone positive feedback and memory}

\author[a]{Franco Herrera}
\author[a]{Sergei Trofimchuk\corref{mycorrespondingauthor}}
\cortext[mycorrespondingauthor]{\hspace{-7mm} {\it e-mails addresses}:   franco.herrera@utalca.cl (F. Herrera); trofimch@inst-mat.utalca.cl (S. Trofimchuk, corresponding author)  \\}
\address[a]{Instituto de Matem\'aticas, Universidad de Talca, Casilla 747, Talca, Chile}

\bigskip

\begin{abstract}
\noindent  We continue to study (see \cite{HTa}) a renewal equation  $\phi(t)=\frak F\phi_t$ proposed in \cite{FBCD, FBCD1} to model  trees growth. This time we are 
considering the case when the per capita reproduction rate $\beta(x)$ is a non-monotone (unimodal) function  of tree's height $x$. Note that  
 the height of some species of trees can impact negatively seed viability, cf. \cite[p. 524]{CO}, in a kind of autogamy depression. 
Similarly to previous works, it is also assumed that the growth rate  $g(x)$ of an individual of height $x$ is a strictly decreasing function.   
 As in \cite{HT,HTa}, we analyse the connection between dynamics of the associated one-dimensional map $F(b)= \frak Fb,$ $b \in \R_+$, and the delayed (hence infinite-dimensional) model $\phi(t)=\frak F\phi_t$.  
 Our key observation is that this model is of monotone positive feedback type since $F$ is strictly increasing on $\R_+$ independently on the monotonicity properties of $\beta$. 
 \end{abstract}

\end{frontmatter}

\newpage

\section{Introduction}  \label{Ex1} 

\noindent In this note, we consider the renewal equation 
\begin{equation}\label{RE}
b(t) =\frak Fb_t:= \int_0^\infty \hspace{-2mm}\beta\left( x_m + \int_0^a \hspace{-1mm} g\left( e^{-\mu(\tau -a)} \int_a^\infty\hspace{-1mm}  e^{-\mu s} b(t-s)\,ds \right) d\tau \right) e^{-\mu a} b(t-a)da,
\end{equation}
where  $b_t(s):=b(t+s), \ s \leq 0, $ and functional  ${\frak F}$ is given by 
 $$
{\frak F}\phi:=\int_0^\infty \beta\left( x_m + \int_0^a g\left( e^{-\mu(\tau -a)} \int_a^\infty e^{-\mu s} \phi(-s)\,ds \right) d\tau \right) e^{-\mu a} \phi(-a)da, $$ This equation was proposed recently in \cite{FBCD, FBCD1} to model  trees growth and presents one of possible approaches to analyse the forest population dynamics, several other  models are mentioned in \cite{FBCD, FBCD1,HTa}.  In equation (\ref{RE}), function  $b(t)$ corresponds to the tree population birth rate at time $t$; continuous strictly decreasing function $g(x)$, $g: [0,+\infty)\to (0, +\infty)$, $g(+\infty)=0$, is the growth rate  of an individual of height $x$;   $\mu >0$ represents the per capita death rate;  globally Lipschitzian function $\beta(x)$, $\beta:[x_m,+\infty)\to [0, +\infty)$, is the per capita reproduction rate (depending only on height $x$).  This model postulates the minimal height $x_m \geq 0$ for all newborn trees. By  \cite{FBCD, FBCD1}, $b(t)$ gives the influx at $x_m$ originating from reproduction by the extant population. To ensure good persistence properties of solutions, it suffices to assume that $\beta(x) > 0$ a.e. on $[x_0,+\infty)$  for some $x_0\geq x_m$ \cite{HTa}.

Clearly, (\ref{RE}) can be considered as an autonomous functional equation with infinite delay  
once  it is  provided with non-negative initial conditions 
$b(s)=\phi(s) \geq 0, \ s \leq 0$, where $\phi$ belongs to the cone $\frak K $ of non-negative measurable functions of the Banach space 
$$
L^1_\rho(\R_{-})= \{\phi: |\phi|_\rho < \infty\}, \quad \mbox{(here we use the notation} \quad |\phi|_q:= \int_{-\infty}^0|\phi(s)|e^{qs}ds),
$$
with $\rho < \mu$. It was proved in \cite{FBCD, FBCD1,HTa} that 
the renewal equation  (\ref{RE}) generates a continuous semiflow 
\mbox{$\frak S: \frak K \times \R_+ \to \frak K $} by  $(\frak S^t \phi)(s) = b(t+s),$ $s \leq 0$. It has a zero 
steady state $b(t) \equiv 0$ while  the positive equilibria of $\frak S^t $ are defined from the equation 
 \begin{equation}\label{ERE}
1= R(b):=\int_0^\infty \beta\left( x_m + \int_0^a g\left( \frac{e^{-\mu\tau}}{\mu}b\right) d\tau \right) e^{-\mu a}da. 
\end{equation}
Note that  equation (\ref{ERE})  has at least one  positive solution $b_*$ if  
$R(0) >1$ and $R(+\infty)= \beta(x_m)/\mu <1$.  If additionally  $\beta$ increases on $\R_+$, then $R(b)$ is strictly decreasing so that solution $b_*$ is unique, cf.  \cite[Theorem 4.1]{FBCD}.  Moreover, this unique positive steady state is locally stable and globally attracting solution of (\ref{RE}), see  
\cite{FBCD, FBCD1,HTa}. The key argument in proofs here is that semiflow $\frak F^t \phi$ is monotone  \cite{Smith} once  $\beta$ is increasing on $\R_+$.

In some cases, however, the assumption of monotonicity of $\beta$ is not realistic. For example,  the height of some species of trees can impact negatively seed viability, cf. \cite[p. 524]{CO}, in a kind of autogamy depression. Thus it would be interesting to study the dynamics in (\ref{RE}) in the case when 
 the per capita reproduction rate $\beta(x)$ is a non-monotone  (for instance, unimodal) function  of tree's height $x$. Now, a general description of this dynamics is known from  \cite{HTa}, where it was proven that under all assumed (except monotonicity of $\beta$) restrictions on $\beta$, $g$ and 
 $0 \leq  \beta(x_m)< \mu$,  $R(0) >1$,  the semiflow  $\frak F^t \phi$ has a compact global attractor $\mathcal A$ 
attracting each non-zero solution. Moreover, each element $\psi$  of $\mathcal A$ satisfies 
$
0 <\theta_1 \leq \psi(s) \leq \theta_2, \quad |\psi'(s)| \leq \theta_3, \ s \leq 0,
$
for some  universal positive numbers $\theta_1 \leq \theta_2, \theta_3$. 

In the sequel we present several analytical and numerical arguments shedding some light on possible dynamical structures of the global attractor $\mathcal A$ without assumption of the monotonicity for $\beta$. As in \cite{HT,HTa}, we will analyse the connection between dynamics of the associated one-dimensional map $F(b)= \frak Fb,$ $b \in \R_+$, and the delayed (hence infinite-dimensional) model $\phi(t)=\frak F\phi_t$.

\section{Monotonicity of the function $F(b)= bR(b)$,  $b \in \R_+$. }  \label{Ex1} 
 \begin{thm}\label{Th1}
Assume all conditions  (except monotonicity of $\beta$) imposed on $\beta$ and $g$ in Section 1. Then function $F(b)= bR(b)$ 
is strictly increasing on $\R_+$.  
\end{thm}
\begin{proof} Assuming that $\phi(s) >0$ a.e. on $\R_-$ and setting  $h(s)=\phi(s)e^{\mu s}$ and $\theta = H(u)=\int_{-\infty}^uh(s)ds$,  we find that $$
{\frak F}\phi=\int_{-\infty}^0 \beta\left( x_m + \int^0_u g( e^{-\mu s}H(u))ds\right)h(u)du =
\int_{0}^{H(0)}\beta\left( x_m + \int^0_{H^{-1}(\theta)} g( e^{-\mu s}\theta)ds\right)d\theta. 
$$
Clearly,  $H(u)$ is strictly increasing and positive on $\R$,  with $H(-\infty)=0$, 
$$
\frac{d}{d\theta}H^{-1}(\theta)= \frac{1}{h(H^{-1}(\theta))} \quad \mbox{a.e. on}\ (0, H(0)). 
$$  
If $\phi_2(s) > \phi_1(s) >0$, almost everywhere on $\R_-$, then 
$$
H_2(u)=\int_{-\infty}^uh_2(s)ds > H_1(u)=\int_{-\infty}^uh_1(s)ds, \quad u \leq 0, 
$$
so that 
$$
H_2^{-1}(\theta) < H_1^{-1}(\theta) \leq 0,    \quad 0 <\theta \leq H_1(0).  
$$
Since $g$ is strictly decreasing, we find that the function 
$$\gamma(\theta, \phi):=x_m + \int^0_{H^{-1}(\theta)} g( e^{-\mu s}\theta)ds =x_m + \int^{\mu^{-1}\ln(1/\theta)}_{H^{-1}(\theta)+\mu^{-1}\ln(1/\theta)} g( e^{-\mu s})ds,$$
 is continuous and strictly decreasing in $\theta$ for each fixed $\phi >0$, 
$$\partial\gamma(\theta, \phi)/\partial\theta= -\frac 1\mu \left(\frac{g(\theta)}{\theta} - g(\theta e^{-\mu H^{-1}(\theta)})\left[\frac{1}{\theta}- \frac{\mu}{h(H^{-1}(\theta))}\right]  \right) <0.$$
Also, for a fixed $\theta$ and $\phi_2(s) > \phi_1(s) >0$, it holds that 
$$\gamma(\theta, \phi_2)> \gamma(\theta, \phi_1).$$
 Thus the equation $u=\gamma(\theta, \phi)$ can be solved with respect to $\theta$, we denote this solution by $\theta=\sigma(u,\phi)$.  Clearly, $\sigma$ is strictly decreasing in the first argument and  is strictly increasing in $\phi$. Therefore $g(\sigma(u,\phi))/\sigma(u,\phi)$ is strictly decreasing in $\phi$
 
 In this way, 
$$ {\frak F}\phi = \mu \int_{x_m}^{+\infty}\beta(u) \left(\frac{g(\theta)}{\theta} - g(\theta e^{-\mu H^{-1}(\theta)})\left[\frac{1}{\theta}- \frac{\mu}{h(H^{-1}(\theta))}\right]\right)^{-1}du, \ \mbox{where}\ \theta=\sigma(u,\phi). 
$$
Particularly, 
$$
F(b)= {\frak F}\phi = \mu \int_{x_m}^{+\infty}\beta(u)\frac{\theta}{g(\theta)}du, \ \mbox{where}\ \theta=\sigma(u,b). 
$$
so that $F(b)$ is a strictly increasing function independently on the monotonicity properties of $\beta$. 
\qed
\end{proof}
Theorem \ref{Th1} suggests to explore a possible connection between dynamical structures for the delayed equation (\ref{RE}) and seemingly simpler equation 
\begin{equation}\label{PF}
b(t)= {\frak G}b_t:= \int_0^\infty f\left(b(t-a-h)\right) e^{-\mu a}da,
\end{equation}
where $F= \mu^{-1}f:\R\to \R$ is a strictly increasing smooth function, $f'(x)>0$ for all $x \in \R$,  having three equilibria  $e_1 < e_2:=0 < e_3$, $F(e_j)= e_j$.  Equation (\ref{PF}) describes bounded solutions from the global attractor for the semiflow generated by the delayed equation 
$$
b'(t)= -\mu b(t)+ f(b(t-h)), \quad \mu >0. 
$$
Specific applications of  such a model can be found  in neural network theory \cite{KWW,KW}. Obviously, each equilibrium for  (\ref{PF}) is determined from the equation 
$$b= {\frak G}b = \mu ^{-1}f(b),
$$
where the right hand side is given by a strictly increasing function. The global attractor $\mathcal A$
 for (\ref{PF}) has relatively simple structure described in \cite{KW} (if $e_2$ is "sufficiently unstable", then $\mathcal A$ is 
a smooth solid 3-dimensional spindle separated by the invariant disk into the basins of attraction toward the tips $e_1$ and $e_3$). See also \cite{KV}
for further results  and references on this topic.

\section{Characteristic equations and principle of linearized stability}  \label{Char1}
Formal linearization of the operator ${\frak F}$ at an equilibrium $b(t) \equiv b$ yields the following action 
$$
{\frak F}'(b)\psi=\int_0^\infty \beta\left( x_m + \int_0^a g\left( \frac{e^{-\mu\tau}b}{\mu} \right) d\tau \right) e^{-\mu a} \psi(-a)da+
$$
 $$
 \int_0^\infty \beta'\left( x_m + \int_0^a g\left( \frac{e^{-\mu\tau}b}{\mu} \right) d\tau \right)\left[g\left( \frac{b}{\mu} \right)- g\left( \frac{e^{-\mu a}b}{\mu} \right)\right]\int_a^\infty e^{-\mu s} \psi(-s)\,ds da=
 $$
  $$\beta(x_m)\int_0^\infty e^{-\mu s} \psi(-s)\,ds+   g\left( \frac{b}{\mu} \right) \int_0^\infty \beta'\left( x_m + \int_0^a g\left( \frac{e^{-\mu\tau}b}{\mu} \right) d\tau \right)\int_a^\infty e^{-\mu s} \psi(-s)\,ds da. 
 $$
 Looking for the exponential eigenfunctions $\psi(s) = \exp(\lambda s)$ of the linearized equation $\psi(t) = {\frak F}'(b)\psi_t$, we obtain the characteristic equation for $\lambda$, cf.  \cite{FBCD, FBCD1}: 
 $$\lambda +(\mu - \beta(x_m))=  g\left( \frac{b}{\mu} \right) \int_0^\infty \beta'\left( x_m + \int_0^a g\left( \frac{e^{-\mu\tau}b}{\mu} \right) d\tau \right)e^{-(\mu+\lambda)a}da =: \chi_b(\lambda). 
 $$
 %or for $z =\lambda+\mu$
% $$z-\beta(x_m)= g\left( \frac{b}{\mu} \right) \int_0^\infty \beta'\left( x_m + \int_0^a g\left( \frac{e^{-\mu\tau}b}{\mu} \right) d\tau \right){e^{-z a}}
 %da =: \zeta_b(z). 
 %$$
 First, we consider the simpler case when $\beta$ is a non-decreasing function and  $\mu> \beta(x_m)$ (so that $F'(b) \leq 1$ at the unique positive equilibrium): 
 \begin{lem} \label{L8} Assume that $\beta'(x_m+s) \geq 0$, $g(s) >0$, for $s \geq 0$ and $\mu - \beta(x_m) >0$. Then the characteristic equation has exactly one real solution $\lambda=\lambda_0$. Moreover, $\lambda_0\leq 0$ and $\Re\lambda_j < \lambda_0$ for each other (complex) solution $\lambda_j$. If $g$ is strictly decreasing on $\R_+$, then $\lambda_0<0$. 
 \end{lem} 
 \begin{proof} The proof is obvious if $\beta'(s) = 0$ a.e. so that we consider the case when $\beta'\not= 0$. Clearly each eigenvalues satisfies the inequality $\lambda_j +(\mu - \beta(x_m)) \geq 0$. Taking into account  opposite types of monotonicity of $\lambda +(\mu - \beta(x_m))$ and $\chi_b(\lambda)$ and the fact that $\chi_b(\lambda) >0$ and $\chi_b(0)= \mu F'(b)-  \beta(x_m)\leq \mu -  \beta(x_m)$, we deduce the existence of a unique real (non-positive) solution $\lambda_0$. 
 
 Assuming that there is another eigenvalue $\lambda_j$ with $\Re \lambda_j \geq \lambda_0$, we get a contradiction: 
 $$\lambda_0 +(\mu - \beta(x_m)) = |\lambda_0 +(\mu - \beta(x_m))| < |\lambda_j +(\mu - \beta(x_m))|\leq   
 $$
 $$
 g\left( \frac{b}{\mu} \right) \int_0^\infty \beta'\left( x_m + \int_0^a g\left( \frac{e^{-\mu\tau}b}{\mu} \right) d\tau \right)e^{-(\mu+\Re\lambda_j)a}da \leq 
 $$
$$ g\left( \frac{b}{\mu} \right) \int_0^\infty \beta'\left( x_m + \int_0^a g\left( \frac{e^{-\mu\tau}b}{\mu} \right) d\tau \right)e^{-(\mu+\lambda_0a)}da =\lambda_0 +(\mu - \beta(x_m)).
$$
Finally, suppose that $g$ is a strictly decreasing and $\lambda_0=0$. 
Then we get the following contradiction: 
$$\mu - \beta(x_m)=   \int_0^\infty \beta'\left( x_m + \int_0^a g\left( \frac{e^{-\mu\tau}b}{\mu} \right) d\tau \right)g\left( \frac{b}{\mu} \right)e^{-\mu a}da < $$
$$ \int_0^\infty \beta'\left( x_m + \int_0^a g\left( \frac{e^{-\mu\tau}b}{\mu} \right) d\tau \right)g\left( \frac{e^{-\mu a}b}{\mu} \right)e^{-\mu a}da =$$
$$
- \beta(x_m) + \mu  \int_0^\infty \beta\left( x_m + \int_0^a g\left( \frac{e^{-\mu\tau}b}{\mu} \right) d\tau \right)e^{-\mu a}da = - \beta(x_m) + \mu.  \quad \Box
$$
 \end{proof}
\begin{remark} Lemma \ref{L8}  simplifies and improves arguments in   \cite[Section 5]{FBCD} (apparently,  the uniqueness statement of this lemma  contradicts to the existence of two negative eigenvalues in the particular case presented in \cite[Subsection 5.1]{FBCD}; 
however,  one of these eigenvalues is false and should be excluded to allow the convergence of integrals in $\chi_b(\lambda)$. 
\end{remark}
 Next, we analyse  the situation when $\beta$ is not necessarily  monotone function and  $\mu> \beta(x_m)$ (we assume that the model possesses at least one positive equilibrium $b$): 
\begin{lem} \label{L9} Assume that  $g>0$ is strictly decreasing Lipschitz continuous function on $\R_+$, $\mu > \beta(x_m)$. Then in the half-plane $\{\Re z \geq - \mu \}$ the characteristic equation has exactly one real solution $\lambda=\lambda_0$ and $\Re\lambda_j < \lambda_0$ for each other (complex) solution $\lambda_j$. Moreover, if $F'(b) < 1$ then $\lambda_0< 0$, $\lambda_0 = 0$ when $F'(b)  = 1$ and  $\lambda_0>0$   if $F'(b) > 1$. 
 \end{lem} 
\begin{proof} Note that 
$$
\lambda +(\mu - \beta(x_m))=  g\left( \frac{b}{\mu} \right)  \int_0^\infty \beta'\left( x_m + \int_0^a g\left( \frac{e^{-\mu\tau}b}{\mu} \right) d\tau \right)e^{-(\mu+\lambda)a}da = 
$$
$$
 g\left( \frac{b}{\mu} \right)  \int_0^\infty \frac{e^{-(\mu+\lambda)a}}{g\left( \frac{e^{-\mu a}b}{\mu} \right) }d\beta\left( x_m + \int_0^a g\left( \frac{e^{-\mu\tau}b}{\mu} \right) d\tau \right) = 
$$
$$
-\beta(x_m) -  g\left( \frac{b}{\mu} \right)  \int_0^\infty \beta\left( x_m + \int_0^a g\left( \frac{e^{-\mu\tau}b}{\mu} \right) d\tau \right) \left[ \frac{e^{-(\mu+\lambda)a}}{g\left( \frac{e^{-\mu a}b}{\mu} \right)}\right]'da. 
$$
Since 
$$
 \left[ \frac{e^{-(\mu+\lambda)a}}{g\left( \frac{e^{-\mu a}b}{\mu} \right)}\right]'= e^{-(\mu+\lambda)a}\frac{-(\mu+\lambda)g\left( \frac{e^{-\mu a}b}{\mu} \right) + {e^{-\mu a}b} g'\left( \frac{e^{-\mu a}b}{\mu} \right)}{\left(g\left( \frac{e^{-\mu a}b}{\mu} \right)\right)^2} <0, 
$$
we obtain the following equivalent form of the characteristic equation: 
$$
1 =  g\left( \frac{b}{\mu} \right)  \int_0^\infty \beta\left( x_m + \int_0^a g\left( \frac{e^{-\mu\tau}b}{\mu} \right) d\tau \right)  \frac{e^{-(\mu+\lambda)a}}{g\left( \frac{e^{-\mu a}b}{\mu} \right)}\left[1-\ \frac{{e^{-\mu a}b} g'\left( \frac{e^{-\mu a}b}{\mu} \right)}{(\mu+\lambda)g\left( \frac{e^{-\mu a}b}{\mu} \right)}\right]da.
$$
The right side of this equation (we will denote it by $\xi_b(\lambda)$, clearly 
$\xi_b(\lambda) = (\chi_b(\lambda)+\beta(x_m))/(\mu+\lambda)$ and $\xi_b(0)= (\chi_b(0) + \beta(x_m))/\mu = F'(b)$) is a strictly decreasing function of $\lambda$ on the interval $(-\mu, + \infty)$, $\xi_b(-\mu^{+})= +\infty$, $\xi_b(+\infty)=0$,  so that the equation  has exactly one real root $\lambda_0$ on this interval.
Furthermore, we see that $\lambda_0$ is a dominating eigenvalue. Indeed if $\lambda_j$ is another eigenvalue with $\Re \lambda_j \geq \lambda_0$ then 
$\lambda_0+\mu < |\lambda_j+\mu|$ and we get a contradiction: 
$$
1 \leq   g\left( \frac{b}{\mu} \right)  \int_0^\infty \beta\left( x_m + \int_0^a g\left( \frac{e^{-\mu\tau}b}{\mu} \right) d\tau \right)  \frac{e^{-(\mu+\Re\lambda_j)a}}{g\left( \frac{e^{-\mu a}b}{\mu} \right)}\left[1+\frac{{e^{-\mu a}b} |g'\left( \frac{e^{-\mu a}b}{\mu} \right)|}{|\mu+\lambda_j|g\left( \frac{e^{-\mu a}b}{\mu} \right)}\right]da <
$$
$$
g\left( \frac{b}{\mu} \right)  \int_0^\infty \beta\left( x_m + \int_0^a g\left( \frac{e^{-\mu\tau}b}{\mu} \right) d\tau \right)  \frac{e^{-(\mu+\lambda_0)a}}{g\left( \frac{e^{-\mu a}b}{\mu} \right)}\left[1+\ \frac{{e^{-\mu a}b} |g'\left( \frac{e^{-\mu a}b}{\mu}\right)|}{|\mu+\lambda_0|g\left( \frac{e^{-\mu a}b}{\mu} \right)}\right]da =1
$$
Finally, since $\xi_b(0)=F'(b)$ we conclude that $\lambda_0 < 0$ when $F'(b)  < 1$,  $\lambda_0 = 0$ when $F'(b)  = 1$ and $\lambda_0 >0$ when $F'(b) > 1$. \qed
\end{proof} 
Lemma \ref{L9} implies that, except for the critical case $F'(b)= 1$, the local stability/instability of an equilibrium $b$ in (\ref{RE}) is equivalent to the local stability/instability of the same equilibrium for one-dimensional system $b \to F(b)$:   

\begin{cor} \label{Cor5} Assume that functions $g: \R \to \R$, $\beta: \R_+ \to \R$ have a bounded and globally Lipschitzian first derivative and $\rho < \mu/5$.  Then an equilibrium $b$ of continuous semiflow 
\mbox{$\frak S: \frak K \times \R_+ \to \frak K $}  is  locally asymptotically stable if  $F'(b) <1$ and is unstable when $F'(b) >1$.  
\end{cor}
\begin{proof} As it was established in \cite[Theorem 3.1]{FBCD}, under  additional assumptions of this corollary,  the functional $ {\frak F}: L^1_\rho(\R_{-})
 \to \R$  is continuously differentiable with bounded derivative.  Then an application of  \cite[Theorem 3.3]{FBCD} (based on Theorem 3.15 in \cite{DG}) completes the proof.   \qed
\end{proof} 
\begin{remark} Note that the inequalities  $F'(b) <1$ and $F'(b) >1$ are equivalent, respectively, to the inequalities $R'(b) <0$ and $R'(b) >0$, where 
$$R'(b)=
-\frac{1}{b} -  \frac{1}{\mu b}g\left( \frac{b}{\mu} \right)  \int_0^\infty \beta\left( x_m + \int_0^a g\left( \frac{e^{-\mu\tau}b}{\mu} \right) d\tau \right) \left[ \frac{e^{-\mu a}}{g\left( \frac{e^{-\mu a}b}{\mu} \right)}\right]'da. 
$$
\end{remark}
\section{Numerical simulations and open problems}  \label{NSOP}
Stability analysis and discussion presented in the previous sections suggest 
the existence of a 'large'  (we believe, open dense)  set $\mathcal O$ in the phase space $\frak K $
 such that solutions with initial data from $\mathcal O$  converge to one of stable equilibria. Thus we can expect that a typical numerical solution of equation (\ref{RE}) is asymptotically constant independently on the monotonicity properties of $\beta$. To check this conclusion and to find out possible configurations for the set of equilibria for (\ref{RE}) with the unimodal fertility function $\beta$, we took $\beta(x) = \alpha xe^{-x}$ (which is used in the  Nicholson blowflies equation and other models, cf.  \cite{MM});  $x_m=0$, $\mu=1$, and $g(x)=pe^{-x}$, for some specific parameters $\alpha, p>0$.  With these choices, our computations exhibit certain diversity in  dynamical patterns for \eqref{RE}.
\begin{figure}[h]\hspace{-7mm}
	\begin{subfigure}{.54\textwidth}
	\includegraphics[width=\textwidth]{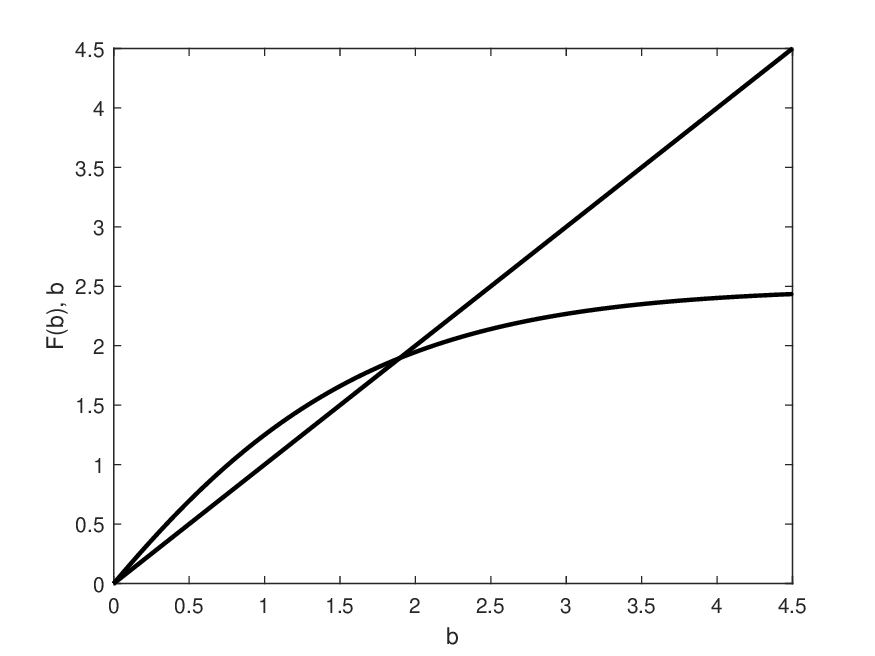}
	\caption{$\alpha=6$ and $p=1$}
	\end{subfigure} \hfill
	\begin{subfigure}{.54\textwidth}
	\includegraphics[width=\textwidth]{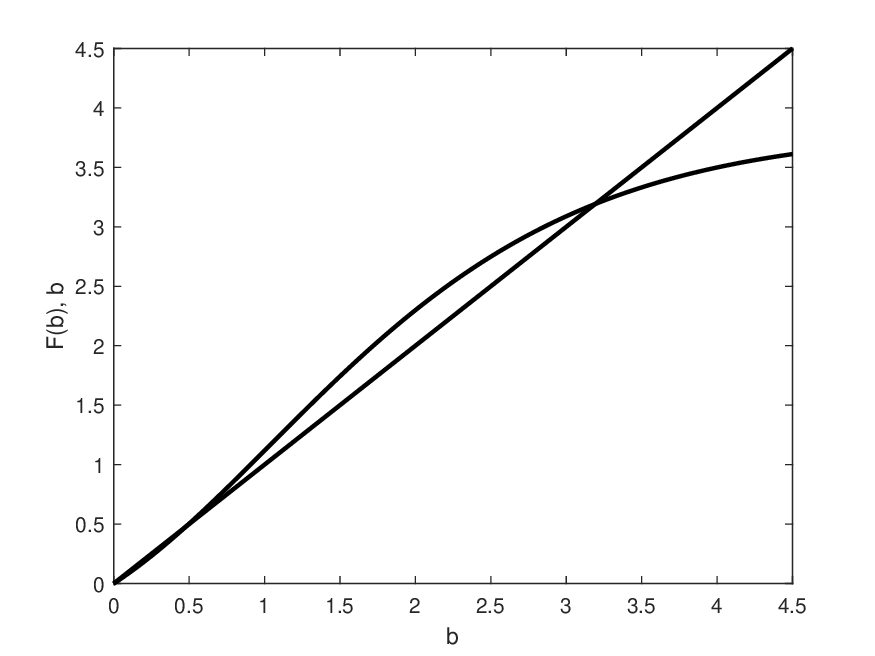}
	\caption{$\alpha=6$ and $p=5$}
	\end{subfigure}
	\caption{Graphics of $y=F(b)$ and $y=b$ for different pairs of parameters $\alpha$ and $p$}
	\label{Fig_F}
\end{figure}
\begin{figure}
	\begin{subfigure}{.45\textwidth}
	\includegraphics[width=\textwidth]{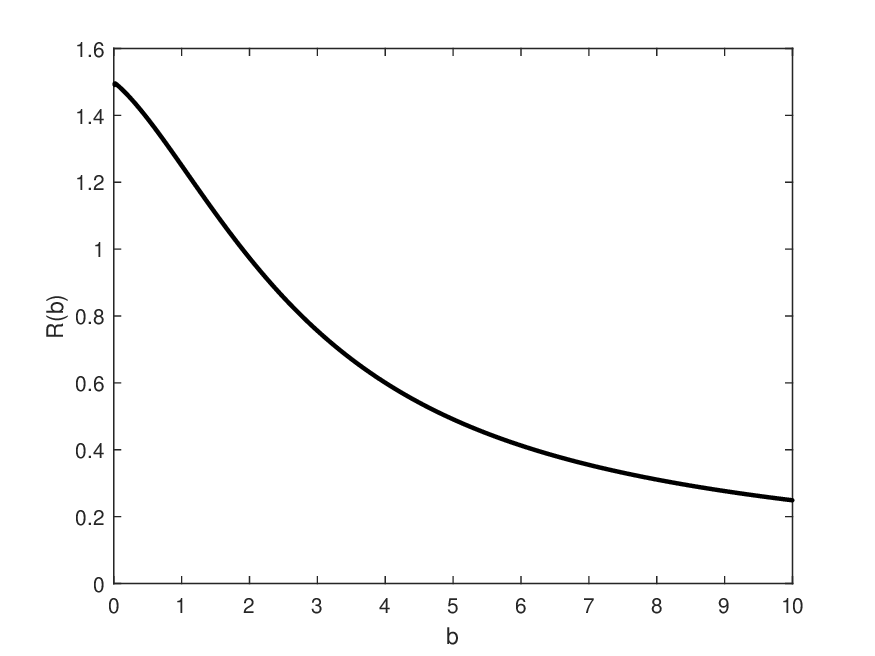}
	\caption{$\alpha=6$ and $p=1$}
	\end{subfigure} \hfill
	\begin{subfigure}{.45\textwidth}
	\includegraphics[width=\textwidth]{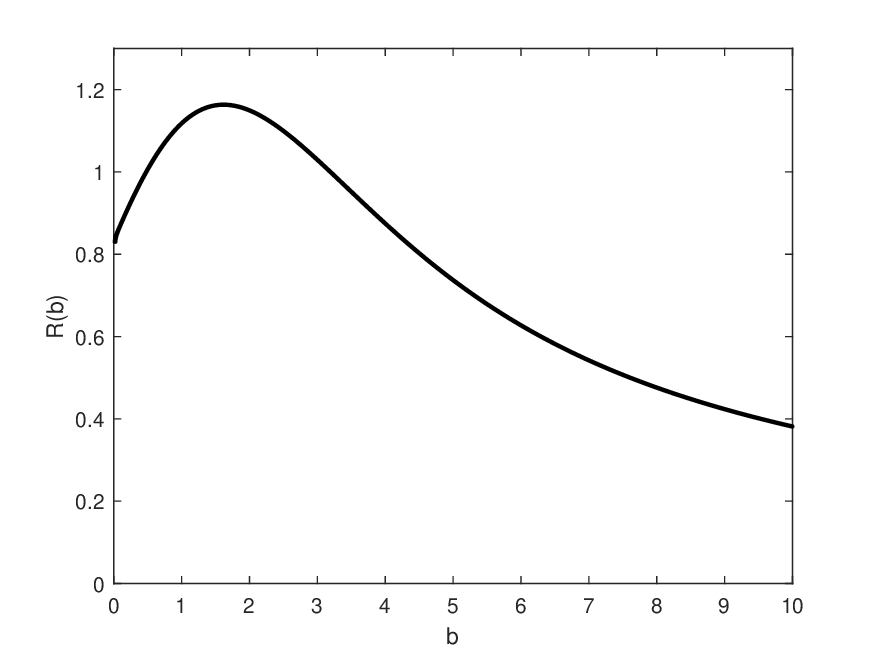}
	\caption{$\alpha=6$ and $p=5$}
	\end{subfigure}
	\caption{Graphics of $R$ for different pairs of parameters $\alpha$ and $p$}
	\label{Fig_R}
		\begin{subfigure}{.45\textwidth}
	\includegraphics[width=\textwidth]{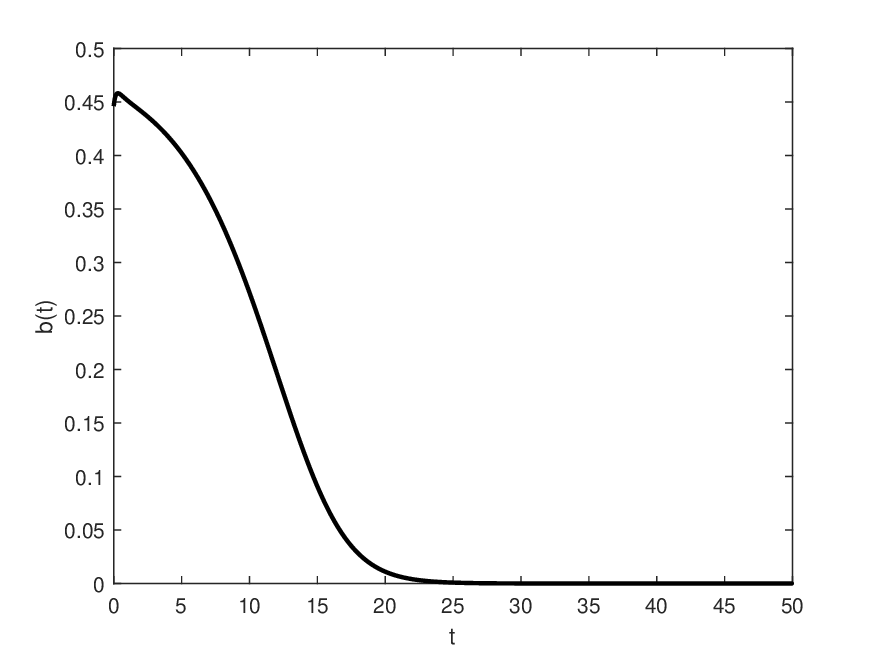}
	\caption{$\phi(a)=0.45$}
	\end{subfigure} \hfill
	\begin{subfigure}{.45\textwidth}
	\includegraphics[width=\textwidth]{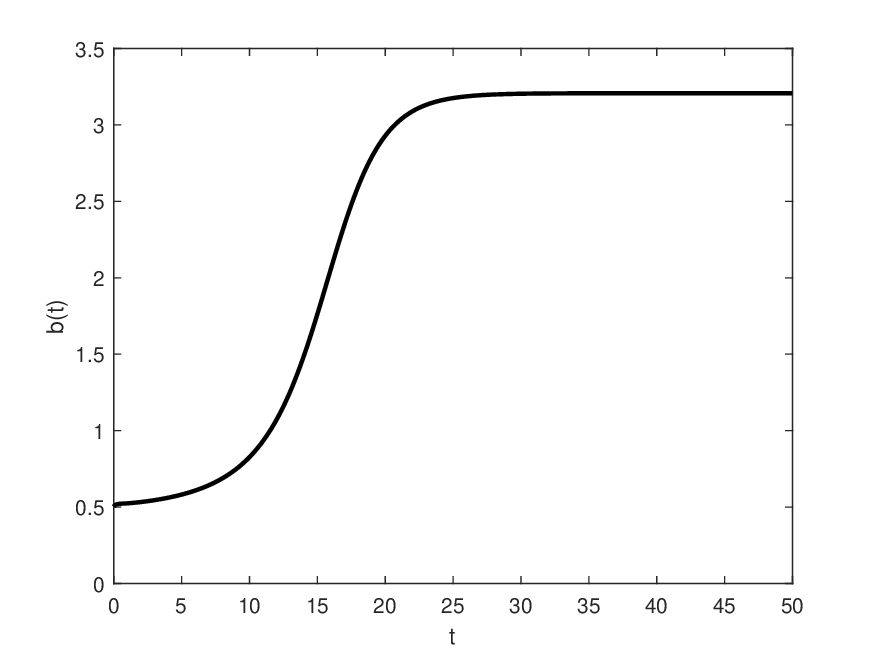}
	\caption{$\phi(a)=0.5$}
	\end{subfigure}
	\caption{Simulations of $b(t)$ using constant initial data $\phi$}
	\label{Fig_b1}
	\begin{subfigure}{.45\textwidth}
	\includegraphics[width=\textwidth]{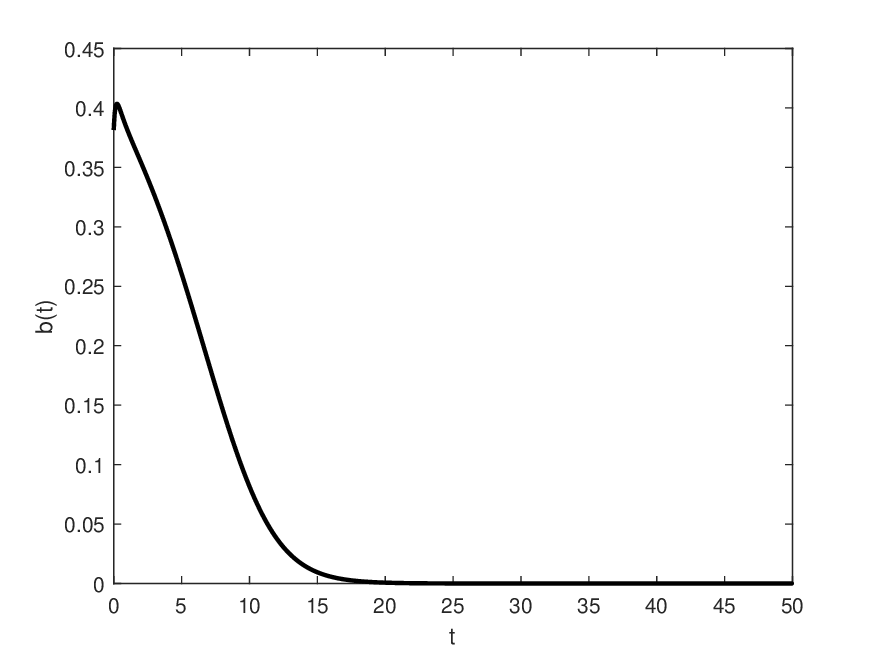}
	\caption{$b_* =0.475$, $\varepsilon=0.2$ and $\omega=1$}
	\end{subfigure} \hfill
	\begin{subfigure}{.45\textwidth}
	\includegraphics[width=\textwidth]{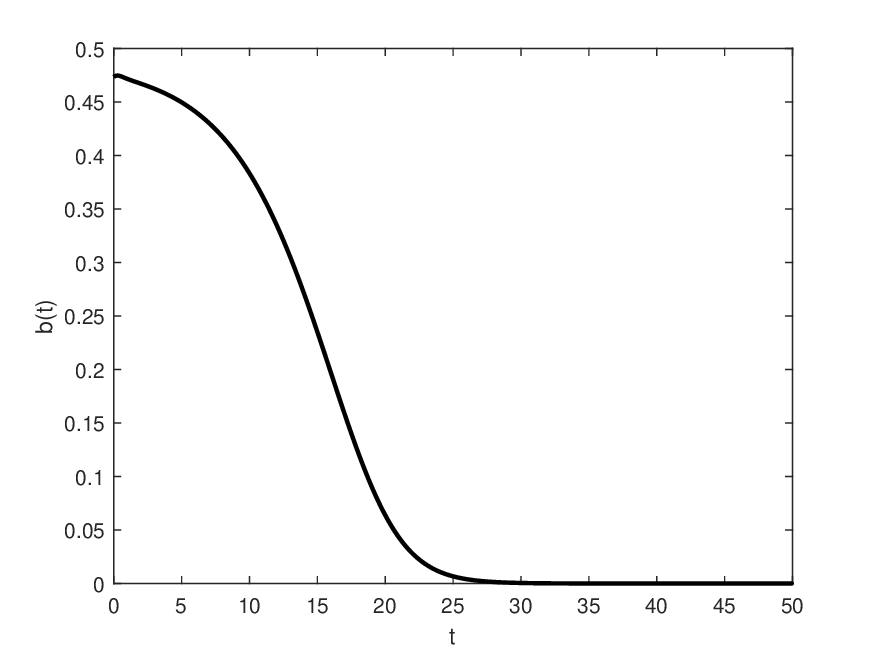}
	\caption{$b_* =0.475$, $\varepsilon=0.2$ and $\omega=30$}
	\end{subfigure}
	\caption{Simulations of $b(t)$ using periodic initial data of the form $\phi(a) = b_* + \varepsilon\sin(\omega a)$}
	\label{Fig_b2}
\end{figure}
First, in full accordance with Theorem \ref{Th1}, Figure \ref{Fig_F} shows that  the function $F$ is strictly increasing even in lack of monotonicity for $\beta$. Nevertheless, the function $R$ is not necessarily monotone and we can observe that for different values of $p$,  $R$ can be a decreasing function as well as  a unimodal one, see Figure \ref{Fig_R}. 
 It is important to mention that these shapes for $R$  do not change essentially even if instead of $g(x)=pe^{-x}$ we take $g$ with other orders of decay (e.g. polynomial) at $+\infty$. 

Interestingly,  regardless the fact $R(0) <1$, equation (\ref{RE})  with $\alpha=6$ and $p=5$ has three equilibria, $b_1=0 < b_2\approx{0.47}<b_3\approx 3.2$. Since $F'(0)=R(0)<1$, $F'(b_2) = 1+ b_2R'(b_2) >1$, $F'(b_3) = 1+ b_3R'(b_3) <1$, we deduce that equilibria $b_1=0$ and $b_3$ are locally asymptotically stable and $b_2$ is unstable (see Corollary \ref{Cor5}).  Our computations show that the linearization of  (\ref{RE})  at $b_2$ has only one (real) eigenvalue with non-negative real part for  $\alpha=6$, $p=5$, therefore we cannot invoke the Hopf bifurcation approach to prove the existence of a periodic orbit for (\ref{RE}).  
% Now, we are interested in stability properties around these equilibria, and also in the possible presence of non-trivial periodic solutions. In order to reach this aims, 
Simulating solutions $b(t)$ of \eqref{RE} for different initial data $\phi$, we took constant and periodic initial functions  close to the intermediate equilibrium $b_2$. Figures 3 and 4 show that these solutions are asymptotically constant. Clearly, the nonlinearity $F(b)$ is bistable here and we observe a kind of the Allee effect for the forest dynamics. 
% This is related to the existence of a purely imaginary root  $\lambda= i\omega$ of characteristic equation at $b_2$, which can be determined  from the following equation: 
%\begin{align*}
%	f(\omega) &:= \left( \mu - \beta(x_m) - g\left( \frac{b_1}{\mu} \right) \int_0^\infty \beta'\left( x_m + \int_0^a g\left( \frac{b_1}{\mu} e^{-\mu\tau} \right) d\tau \right) e^{-\mu a}\cos(\omega a) da \right)^2\\
%	& \quad + \left( \omega + g\left( \frac{b_1}{\mu} \right) \Big[ \int_0^\infty \beta'\left( x_m + \int_0^a g\left( \frac{b_1}{\mu} e^{-\mu\tau} \right) d\tau \right) e^{-\mu a}\sin(\omega a) da \right)^2 = 0.
%\end{align*}
%Our computations shows that when $\alpha=$ a pair of eigenvalues crosses the imaginary axis at $\omega=$ so that it is possible to use the Hopf bifurcation approach to prove the existence of an unstable periodic orbit for (\ref{RE}). 

%\vspace{2mm}

%We conclude this note formulating several {open questions} about the dynamics in (\ref{RE}): 

%1) How many positive equilibria can have (\ref{RE}) for $\beta(x) = \alpha xe^{-x}$? In general, for a general unimodal fertility rate $\beta$? 

%2) Suppose that (\ref{RE}) has a unique positive equilibrium. Should it be globally asymptotically stable?

%3) In the case of three equilibria (e.g. as on Figure 1b), describe the fine structure of the global attractor, cf. \cite{KW}. 
%Particularly, study the possibility of the Hopf bifurcation at $b_2$ and the existence and uniqueness of a non-trivial periodic solution for (\ref{RE}). 

\vspace{-3mm}

\section*{Acknowledgments}    \noindent   This research was supported in part by the projects FONDECYT 1231169 and AMSUD220002 (ANID, Chile). The first author was also supported by  ANID-Subdirecci\'on de Capital Humano/Doctorado Nacional/2024--21240616. 
 The second author gratefully acknowledges the hospitality and the support of the MATRIX Institute allowing him to participate in the Research Program "Delay Differential Equations and Their Applications" (Australia, Creswick,  12 --20 Dec 2023).

\end{document}